\documentclass[a4paper,10pt]{amsart}
\usepackage[utf8]{inputenc}
\usepackage{amsxtra}
\usepackage{amsopn}
\usepackage{amsmath,amsthm,amssymb}
\usepackage{amscd}
\usepackage{color}
\usepackage{amsfonts}
\usepackage{latexsym}
\usepackage{verbatim}
\usepackage{color}
\usepackage{enumerate}
\usepackage{calrsfs}

\theoremstyle{plain}
\newtheorem{theorem}{Theorem}[section]
\newtheorem*{theorem*}{Theorem}

\newtheorem{lemma}[theorem]{Lemma}
\newtheorem{prop}[theorem]{Proposition}
\newtheorem{cor}[theorem]{Corollary}
\newtheorem{rem}[theorem]{Remark}
\newtheorem{ex}[theorem]{Example}
\newtheorem*{mt*}{Main Theorem}


\newcommand\C{{\mathbb C}}

\newcommand{\del}{\partial}
\newcommand{\delbar}{\bar{\del}}

\newcommand{\pa}[1]{\left( #1 \right)}

\DeclareMathOperator{\id}{id}

\title[decompositions of $\delbar$-harmonic forms on almost-K\"ahler manifolds]{Primitive decompositions of Dolbeault harmonic forms on compact almost-K\"ahler manifolds}
\author[Andrea Cattaneo, Nicoletta Tardini and Adriano Tomassini]{Andrea Cattaneo, Nicoletta Tardini and Adriano Tomassini}
\address{Dipartimento di Scienze Matematiche, Fisiche e Informatiche\\
Unit\`{a} di Matematica e Informatica,
Universit\`{a} degli Studi di Parma\\
Parco Area delle Scienze 53/A, 43124 \\
Parma, Italy}
\email{andrea.cattaneo@unipr.it}
\email{nicoletta.tardini@gmail.com}
\email{nicoletta.tardini@unipr.it}
\email{adriano.tomassini@unipr.it}

\keywords{almost-complex; Hermitian metric; harmonic form; primitive form}
\thanks{\newline 
The first author is partially supported by GNSAGA of INdAM. The second author is partially supported by GNSAGA of INdAM and has financially been supported by the Programme ``FIL-Quota Incentivante'' of University of Parma and co-sponsored by Fondazione Cariparma. The third author is partially supported by the Project PRIN 2017 ``Real and Complex Manifolds: Topology, Geometry and holomorphic dynamics'' and by GNSAGA of INdAM}
\subjclass[2010]{53C15; 58A14; 58J05}

\begin{document}

\begin{abstract}
Let $(X,J,g,\omega)$ be a compact $2n$-dimensional almost-K\"ahler manifold. We prove primitive decompositions of $\del$-, $\delbar$-harmonic forms on $X$ in bidegree $(1,1)$ and $(n-1,n-1)$ (such bidegrees appear to be optimal). We provide examples showing that in bidegree $(1,1)$ the $\del$- and $\delbar$-decompositions differ.
\end{abstract}

\maketitle

\section{Introduction}

In complex geometry the Dolbeault cohomology plays a fundamental role in the study of complex manifolds and a classical way to compute it on compact complex manifolds is through the use of the associated spaces of harmonic forms. More precisely, if $X$ is a complex manifold, then the exterior derivative $d$ splits as $\del+\delbar$ and such operators satisfy $\delbar^2=\del^2=\del\delbar+\delbar\del=0$. Hence, one can define the Dolbeault cohomology and its conjugate as
$$
H^{\bullet,\bullet}_{\delbar}(X):=
\frac{\text{Ker}\,\delbar}{\text{Im}\,\delbar}\,,\qquad
H^{\bullet,\bullet}_{\del}(X):=
\frac{\text{Ker}\,\del}{\text{Im}\,\del}.
$$
If $X$ is compact and we fix an Hermitian metric, then it turns out that these spaces are isomorphic to the kernel of two suitable elliptic operators, $\Delta_{\delbar}$ and $\Delta_{\del}$ respectively. More precisely, denoting with  $\mathcal{H}^{\bullet,\bullet}_{\delbar}(X)$ and $\mathcal{H}^{\bullet,\bullet}_{\del}(X)$ the spaces of harmonic forms, they have a cohomological meaning, namely
$$
H^{\bullet,\bullet}_{\delbar}(X)\simeq
\mathcal{H}^{\bullet,\bullet}_{\delbar}(X)
\,,\qquad
H^{\bullet,\bullet}_{\del}(X)\simeq
\mathcal{H}^{\bullet,\bullet}_{\del}(X)
$$
and in particular their dimensions are holomorphic invariants.\\
Moreover, if the Hermitian metric is K\"ahler then by the K\"ahler identities it turns out that $\Delta_{\delbar}=\Delta_{\del}$ and in particular
$$
\mathcal{H}^{\bullet,\bullet}_{\delbar}(X)=\mathcal{H}^{\bullet,\bullet}_{\del}(X)
$$
giving therefore isomorphisms for the respective cohomologies, namely
$$
H^{\bullet,\bullet}_{\delbar}(X)\simeq
H^{\bullet,\bullet}_{\del}(X)\,.
$$
The integrability assumption on the complex structure is crucial in the proof of all these results.\\
Furthermore, a remarkable feature of K\"ahler geometry is that the primitive decomposition of differential forms passes to cohomology and leads to a primitive decomposition of de Rham cohomology (see e.g., \cite{weil}). K\"ahler geometry is at the crossroad of complex and symplectic geometry. From the symplectic point of view we recall that in \cite{tseng-yau} Tseng and Yau introduced natural cohomologies on (compact) symplectic manifolds, involving the symplectic co-differential and the exterior derivative, proving a primitive decomposition for them. \\
If $J$ is a non-integrable almost-complex structure on a $2n$-dimensional smooth manifold $X$, then the exterior derivative splits as $\mu+\del+\delbar+\bar\mu$ and in particular $\delbar^2\neq 0$. Hence, the standard Dolbeault cohomology and its conjugate are not well-defined. 
Recently, Cirici and Wilson in \cite{cirici-wilson-1}
gave a definition for the Dolbeault cohomology in the non-integrable setting considering also the operator $\bar\mu$ together with $\delbar$. Such cohomology groups might be infinite-dimensional on compact almost-complex manifolds as shown by \cite{coelho-placini-stelzig}.\\
On the other side, fixing an almost-Hermitian metric $g$ on $(X,J)$ one can develop an Hodge theory for harmonic forms on $(X,J,g)$ without a cohomological counterpart. More precisely, setting, similarly to the integrable case,
$$
\Delta_{\delbar}:=\delbar\delbar^*+\delbar^*\delbar\,,\qquad
\Delta_{\del}:=\del\del^*+\del^*\del
$$
it turns out that they are elliptic selfadjoint differential operators. Therefore, if $X$ is compact their kernels, denoted again with $\mathcal{H}^{\bullet,\bullet}_{\delbar}(X)$ and $\mathcal{H}^{\bullet,\bullet}_{\del}(X)$, are finite dimensional complex vector spaces. Holt and Zhang in \cite{holt-zhang} answered to a question of Kodaira and Spencer \cite{hirzebruch} showing that, contrarily to the complex case, the dimensions of the spaces of $\delbar$-harmonic $(0,1)$-forms on a $4$-dimensional manifold depend on the metric. Indeed they construct on the Kodaira-Thurston manifold an almost-complex structure that, with respect to different almost-Hermitian metrics, has varying $\text{dim}\,\mathcal{H}^{0,1}_{\delbar}$.
With different techniques in \cite{tardini-tomassini-dim4} it was shown that 
also the dimension of the space of $\delbar$-harmonic $(1,1)$-forms depend on the metric on $4$-dimensional manifolds (for other results in this direction see \cite{PT4} and \cite{holt}).\\
We note that explicit computations of $\delbar$-harmonic forms is a difficult task and not much is known in higher dimension (see \cite{tardini-tomassini-dim6}, \cite{cattaneo-nannicini-tomassini1}, \cite{CNT} for some detailed computations). \\
In the present paper we study the validity of primitive decompositions on compact almost-K\"ahler manifolds in any dimension.
 More precisely, in Propositions \ref{prop1: decomposition}, \ref{prop2: decomposition}, Theorem \ref{thm: decomposition} and Corollary \ref{cor: decomposition} we prove   on compact almost-K\"ahler $2n$-dimensional manifolds, primitive decompositions for $\delbar$- and $\del$-harmonic forms in bidegrees $(p,0)$, $(0,q)$, $(1,1)$, $(n,n-p)$, $(n-q,n)$ and $(n-1,n-1)$, with $p,q\leq n$. One cannot hope to have such decompositions for any bidegree as shown in Example \ref{ex:no-2-1-decomposition}. For similar results in the case of Bott-Chern harmonic forms we refer to \cite{piovani-tardini}.\\
We notice that, even though the metric is almost-K\"ahler the decompositions of $\delbar$- and $\del$-harmonic forms might differ. Indeed,  in Section \ref{section:relations-laplacians} we show explicitly that, differently from the K\"ahler case, one can have
$\Delta_{\delbar}\neq \Delta_{\del}$ and also
$$
\mathcal{H}^{1,1}_{\delbar}(X)\neq\mathcal{H}^{1,1}_{\del}(X)\,.
$$
We observe that a key ingredient in the proof of the results in \cite{tardini-tomassini-dim4} (see also \cite{holt-zhang}) is indeed the primitive decomposition of $\delbar$-harmonic $(1,1)$-forms on $4$-dimensional manifolds. In fact, in this dimension in Proposition \ref{prop:dim4-equality} we prove the general equality $\mathcal{H}^{1,1}_{\delbar}(X)=\mathcal{H}^{1,1}_{\del}(X)$.\\
All the examples we present are nilmanifolds, of dimensions $6$ and $8$, endowed with possibly non left-invariant almost-K\"ahler structures.\\
We recall that if one wants to mimic and recover all the K\"ahler identities the proper operator to consider is $\bar\delta:=\delbar+\mu$ (see \cite{cirici-wilson-2}, \cite{tardini-tomassini-diff-operators}). However, considering just
the operator $\delbar$ on almost-K\"ahler manifolds we are able to see how genuinely almost-K\"ahler manifolds differ from K\"ahler ones. More precisely, the study of the kernel of $\Delta_{\delbar}$ enlightens the purely almost-complex properties.

\medskip
\noindent{\sl Acknowledgments.}  The authors would like to thank Riccardo Piovani for several discussions on the subject. The authors express their gratitude to the anonymous referee for the improvements made after their review.

\section{Preliminaries}
\label{preliminaries}
In this Section we recall some basic facts about almost-complex and almost-Hermitian manifolds and fix some notations.
Let $X$ be a smooth manifold of dimension $2n$ and let $J$ be an almost-complex structure on $X$, namely a $(1,1)$-tensor on $X$ such that $J^2=-\id$. Then $J$ induces on the space of forms $A^\bullet(X)$ a natural bigrading, namely
$$
A^\bullet(X)=\bigoplus_{p+q=\bullet}A^{p,q}(X)\,.
$$
Accordingly, the exterior derivative $d$ splits into four operators
$$
d:A^{p,q}(X)\to A^{p+2,q-1}(X)\oplus A^{p+1,q}(X)\oplus A^{p,q+1}(X)\oplus A^{p-1,q+2}(X)
$$
$$
d=\mu+\del+\delbar+\bar\mu\,,
$$
where $\mu$ and $\bar\mu$ are differential operators that are linear over functions. In particular, they are related to the Nijenhuis tensor $N_J$ by
$$
\left(\mu\alpha+\bar\mu\alpha\right)(u,v)=\frac{1}{4} \alpha\left(N_J(u,v)\right)
$$
where $\alpha\in A^1(X)$. Hence, $J$ is integrable, that is $J$ induces a complex structure on $X$, if and only if $\mu=\bar\mu=0$.\\
In general, since $d^2=0$ one has
$$
\left\lbrace
\begin{array}{lcl}
\mu^2 & =& 0\\
\mu\del+\del\mu & = & 0\\
\del^2+\mu\delbar+\delbar\mu & = & 0\\
\del\delbar+\delbar\del+\mu\bar\mu+\bar\mu\mu & = & 0\\
\delbar^2+\bar\mu\del+\del\bar\mu & = & 0\\
\bar\mu\delbar+\delbar\bar\mu & = & 0\\
\bar\mu^2 & =& 0
\end{array}
\right.\,.
$$
In particular, $\delbar^2\neq 0$ and so the Dolbeault cohomology of $X$
$$
H^{\bullet,\bullet}_{\delbar}(X):=
\frac{\text{Ker}\,\delbar}{\text{Im}\,\delbar}
$$
is well defined if and only if $J$ is integrable. The same holds for the operator $\del$.\\
If $g$ is an Hermitian metric on $(X,J)$ with fundamental form $\omega$ and $*$ is the associated $\mathbb{C}$-linear Hodge-$*$-operator, one can consider the adjoint operators
\begin{equation*}
d^*=-*d*,\ \ \
\mu^*=-*\bar\mu*,\ \ \ \del^*=-*\delbar*,\ \ \ \delbar^*=-*\del*,\ \ \ \bar\mu^*=-*\mu*,
\end{equation*}
and for $D\in\left\lbrace d,\del,\delbar,\mu,\bar\mu\right\rbrace$ one defines the
associated Laplacians
$$
\Delta_D:=DD^*+D^*D
$$
and we will denote the kernel by
$$
\mathcal{H}^{p,q}_{D}(X):=\text{Ker}\,\Delta_{D_{\vert A^{p,q}(X)}}\,.
$$
These spaces will be called the spaces of $D$-harmonic forms.
The operators $\Delta_{\delbar}$ and $\Delta_{\del}$
are second order, elliptic, differential operators and in particular, if $X$ is compact, the associated spaces of harmonic forms are finite-dimensional and their dimensions will be denoted by $h^{p,q}_{\delbar}$ and $h^{p,q}_{\del}$.\\
If $X$ is compact, then we easily deduce the following relations for a $(p,q)$-form $\alpha$,
\begin{equation*}
\begin{cases}
\Delta_{\del}\,\alpha=0\ &\iff\ \del\alpha=0,\ \delbar*\alpha=0,\\
\Delta_{\delbar}\,\alpha=0\ &\iff\ \delbar\alpha=0,\ \del*\alpha=0,
\end{cases}
\end{equation*}
which characterize the spaces of harmonic forms.

\section{Primitive decompositions of Dolbeault harmonic forms}\label{main}
Let $(X,J,g,\omega)$ be a $2n$-dimensional almost-Hermitian manifold. 
We denote with
$$
L:\Lambda^kX\to\Lambda^{k+2}X\,,\quad \alpha\mapsto\omega\wedge\alpha
$$
the Lefschetz operator and with
$$
\Lambda:\Lambda^kX\to\Lambda^{k-2}X\,,\quad \Lambda=-*L*
$$
its dual.
A $k$-form $\alpha_k$ on $X$, for $k\leq n$, is said to be {\em primitive} if $\Lambda\alpha_k=0$, or equivalently $L^{n-k+1}\alpha_k=0$. Then, the following vector bundle decomposition holds (see e.g., \cite{weil}) 
$$
\Lambda^kX=\bigoplus_{r\geq\max(k-n,0)}L^r(P^{k-2r}X),
$$
where 
$$
P^{s}X:=\ker\big(\Lambda:\Lambda^{s}X\to\Lambda^{s-2}X\big)
$$
is the bundle of $s$-primitive forms. Accordingly, given any $k$-form $\alpha_k$ on $X$, we can write 
\begin{equation}\label{primitive-bundle-decomposition}
\alpha_k=\sum_{r\geq\max(k-n,0)}\frac{1}{r!}L^r\beta_{k-2r},
\end{equation}
where $\beta_{k-2r}\in\Gamma(P^{k-2r}X)$, that is 
$$\Lambda\beta_{k-2r} =0,
$$ 
or equivalently 
$$
L^{n-k+2r+1}\beta_{k-2r}=0.
$$
Furthermore, the decomposition above is compatible with the bidegree decomposition on the bundle of complex $k$-forms $\Lambda_\C^kX$ induced by $J$, that is 
$$
P_\C^kX=\bigoplus_{p+q=k}P^{p,q}X,
$$
where 
$$
P^{p,q}X=P^k_\C X\cap\Lambda^{p,q}X
$$
For any given $\beta_k\in P^kX$, we have the following formula (cf. \cite[p. 23, Th\'eor\`eme 2]{weil})
\begin{equation}\label{*-primitive}
*L^r\beta_k=(-1)^{\frac{k(k+1)}{2}}\frac{r!}{(n-k-r)!}L^{n-k-r}J\beta_k.
\end{equation}
In the sequel we will denote $P^{\bullet}=P^{\bullet}X$ and so on.\\
We recall that by \cite[Corollary 5.4]{cirici-wilson-2} such decompositions in primitive forms pass to the spaces of $d$-harmonic forms whenever there exists an almost-K\"ahler metric. More precisely, if $(X,J,\omega)$ is a compact $2n$-dimensional almost-K\"ahler manifold, then, for every $p,q$,
$$
\mathcal{H}_d^{p,q}(X)=\bigoplus_{r\geq\max(k-n,0)}L^r(
\mathcal{H}_d^{p-r,q-r}(X)\cap P^{p-r,q-r}).
$$
In fact, this holds also for the spaces of $\bar\delta$- and $\delta$-harmonic forms introduced in \cite{tardini-tomassini-diff-operators}, where $\bar\delta:=\delbar+\mu$ and $\delta:=\del+\bar\mu$. Indeed, by \cite[Proposition 6.2, Theorem 6.7]{tardini-tomassini-diff-operators}, one has $$
\mathcal{H}_d^{p,q}(X)=\mathcal{H}_{\bar\delta}^{p,q}(X)=
\mathcal{H}_{\delta}^{p,q}(X)\,.
$$
In the following, we are going to study such decompositions for $\delbar$-harmonic forms. First, notice that, since $(p,0)$-forms and $(0,q)$-forms are trivially primitive, we immediately derive the following
\begin{prop}\label{prop1: decomposition}
Let $(X,J,g,\omega)$ be a compact $2n$-dimensional almost-Hermitian manifold (with $n \geq 2$). Then the following decompositions hold for every $p,q\leq n$
\begin{align*}
\mathcal{H}^{p,0}_{\delbar}&=\mathcal{H}^{p,0}_{\delbar}\cap P^{p,0},\ \ \ &\mathcal{H}^{0,q}_{\delbar}&=\mathcal{H}^{0,q}_{\delbar}\cap P^{0,q},\\
\mathcal{H}^{p,0}_{\del}&=\mathcal{H}^{p,0}_{\del}\cap P^{p,0},\ \ \ &\mathcal{H}^{0,q}_{\del}&=\mathcal{H}^{0,q}_{\del}\cap P^{0,q}.
\end{align*}
\end{prop}
By applying to such decompositions the Hodge-$*$-operator and formula (\ref{*-primitive}), we obtain the following
\begin{prop}\label{prop2: decomposition}
Let $(X,J,g,\omega)$ be a compact $2n$-dimensional almost-Hermitian manifold (with $n \geq 2$). Then the following decompositions hold for every $p,q\leq n$
\begin{align*}
\mathcal{H}^{n,n-p}_{\delbar}&=L^{n-p}\left(\mathcal{H}^{p,0}_{\del}\cap P^{p,0}\right),\ \ \ &\mathcal{H}^{n-q,n}_{\delbar}&=L^{n-q}\left(\mathcal{H}^{0,q}_{\del}\cap P^{0,q}\right),\\
\mathcal{H}^{n,n-p}_{\del}&=L^{n-p}\left(\mathcal{H}^{p,0}_{\delbar}\cap P^{p,0}\right),\ \ \ &\mathcal{H}^{n-q,n}_{\del}&=L^{n-q}\left(\mathcal{H}^{0,q}_{\delbar}\cap P^{0,q}\right).
\end{align*}
\end{prop}

As a consequence, we derive the following
\begin{cor}
Let $(X,J,g,\omega)$ be a compact $2n$-dimensional almost-Hermitian manifold (with $n \geq 2$). Then,
$$
\mathcal{H}^{n,0}_{\delbar}=\mathcal{H}^{n,0}_{\del}
\qquad
\text{and}
\qquad
\mathcal{H}^{0,n}_{\delbar}=\mathcal{H}^{0,n}_{\del}.
$$
\end{cor}
\begin{proof}
This follows taking $p=n$ and $q=n$ in Proposition \ref{prop2: decomposition}. Otherwise, it can be proved directly. Indeed, let $\alpha$ be a $(n,0)$-form (the case $(0,n)$ is similar), then $\alpha$ is primitive and
by Formula (\ref{*-primitive}), $*\alpha=c_n\,\alpha$, with $c_n\neq 0$ a constant depending only on the dimension of $X$. Therefore, for bidegree reasons,
$$
\alpha\in\mathcal{H}^{n,0}_{\delbar}\iff
\delbar\alpha=0\iff
\delbar*\alpha=0\iff
\alpha\in\mathcal{H}^{n,0}_{\del}.
$$
\end{proof}

We show now that primitive decompositions hold also in other suitable degrees as soon as we assume the existence of an almost-K\"ahler metric.

\begin{theorem}\label{thm: decomposition}
Let $(X,J,g,\omega)$ be a compact $2n$-dimensional almost-K\"ahler manifold (with $n \geq 2$). Then the following decompositions hold:
\[
\mathcal{H}^{1,1}_{\delbar}=\C\cdot\omega \oplus \pa{ \mathcal{H}^{1,1}_{\delbar}\cap P^{1,1}}
\]
\end{theorem}
\begin{proof} Let $\alpha_{1, 1} \in A^{1, 1}(X)$. Then the primitive decomposition \eqref{primitive-bundle-decomposition} reads as 
\begin{equation}\label{primitive-forms-decomposition}
 \alpha_{1, 1} = \beta_{1, 1} + \beta \omega,
\end{equation}
where 
\begin{equation}
\beta_{1,1}\in A^{1,1}(X),\qquad \beta_{1,1}\wedge\omega^{n-1}=0,\quad\beta\in\mathcal{C}^\infty(X;\C) 
\end{equation}
The form $\alpha_{1, 1}$ belongs to $\mathcal{H}^{1, 1}_{\delbar}$ if and only if $\alpha_{1, 1}$ satisfies the following equations 
\begin{equation}\label{harmonic1-1}
\delbar\alpha_{1, 1}=0, \quad
\del*\alpha_{1, 1}=0.
\end{equation}

By \eqref{*-primitive} we compute 
\begin{equation}\label{*-primitive-n-1}
*\alpha_{1, 1} = -\frac{1}{(n - 2)!} \beta_{1, 1} \wedge \omega^{n - 2} + \beta \frac{1}{(n - 1)!}\omega^{n - 1}.
\end{equation}
Therefore, by \eqref{primitive-forms-decomposition}, \eqref{*-primitive-n-1}, taking into account that $g$ is almost-K\"ahler, equations \eqref{harmonic1-1} are equivalent to 
\begin{equation}\label{system-harmonic}
\left\{ 
\begin{array}{rl}
 \delbar \beta_{1, 1} + \delbar \beta \wedge \omega & = 0\\
 -\frac{1}{(n - 2)!} \del \beta_{1, 1} \wedge \omega^{n - 2} + \del \beta \wedge \frac{1}{(n - 1)!}\omega^{n - 1}& = 0.
 \end{array}
\right.
\end{equation}
After multiplying the first equation by $\omega^{n - 2}$, and the second by $(n - 2)!$, we obtain
$$
\left\{ 
\begin{array}{rl}
 \delbar \beta_{1, 1} \wedge \omega^{n - 2} + \delbar \beta \wedge \omega^{n - 1} & = 0\\
 -\del \beta_{1, 1} \wedge \omega^{n - 2} + \frac{1}{n - 1}\del \beta \wedge \omega^{n - 1} & = 0,
 \end{array}
\right.
$$
and taking the sum of the last two equations we obtain
$$
(\delbar \beta_{1, 1} - \del \beta_{1, 1}) \wedge \omega^{n - 2} + (\delbar \beta + \frac{1}{n - 1} \del \beta) \wedge    \omega^{n - 1} = 0.
$$
By definition, we have
$$
d^c=i(\delbar-\del+\mu-\overline{\mu}),
$$
where $\vert\mu\vert=(2,-1)$, $\vert\overline{\mu}\vert=(-1,2)$. Consequently, the last equation can be written as 
$$
(\delbar \beta + \frac{1}{n - 1}\del \beta) \wedge \omega^{n - 1} = i d^c \beta_{1, 1} \wedge \omega^{n - 2}.
$$
Applying $-id^c$ to both sides of the above equation, we obtain
$$
\left[ (\delbar - \del + \mu - \bar{\mu}) (\delbar \beta + \frac{1}{n - 1} \del \beta) \right] \wedge \omega^{n - 1} = 0,
$$
which yields
$$
(\frac{1}{n - 1}+ 1) \del \delbar \beta \wedge \omega^{n - 1} = 0,
$$
since $\del \delbar + \delbar \del = 0$ on functions and the other contributions vanish by bidegree reasons when we take take the wedge product with $\omega^{n-1}$. Therefore,
$$
\del \delbar \pa{\beta \cdot \omega^{n-1}} = 0
$$
from which we derive that $\beta \equiv \beta_0 \in \C$ is constant (see for instance \cite{gauduchon}, \cite[Theorem 10]{angella-istrati-otiman-tardini} or \cite[Proposition 3.4]{tardini-tomassini-dim4} for the $4$-dimensional case). Hence 
\[\alpha_{1, 1} = \beta_{1, 1} + \beta_0 \omega\]
so that from (\ref{system-harmonic}) or from
$$
\begin{array}{lllll}
\delbar\beta_{1,1} & = & \delbar \alpha_{1, 1} - \delbar(\beta_0 \omega) & = & 0\\
\del *\beta_{1,1} & = & \del*\alpha_{1, 1} - \del *(\beta_0 \omega) & = & 0,
\end{array}
$$
we have $\beta \in \mathcal{H}^{1, 1}_{\delbar}$ and $\beta_{1,1}$ is primitive. This proves that
$$
\mathcal{H}^{1, 1}_{\delbar} \subset \C \cdot \omega \oplus \pa{ \mathcal{H}^{1, 1}_{\delbar} \cap P^{1, 1}}.
$$

Conversely, if $\alpha_{1, 1} = \beta_0 \omega + \beta_{1, 1}$, with $\beta_0 \in \C$ and $\beta_{1,1} \in \mathcal{H}^{1, 1}_{\delbar} \cap P^{1, 1}$ we easily conclude that $\del *\alpha_{1, 1} = 0$ and that $\delbar \alpha_{1, 1} = 0$. The decomposition i) is thus proved.
\end{proof}


As a consequence we obtain the following primitive decompositions.

\begin{cor}\label{cor: decomposition}
Let $(X,J,g,\omega)$ be a compact $2n$-dimensional almost-K\"ahler manifold (with $n \geq 2$). Then the following decompositions hold:
\begin{enumerate}[i.]
 \item $\displaystyle \mathcal{H}^{1, 1}_{\del} = \C \cdot \omega \oplus \pa{ \mathcal{H}^{1, 1}_{\del}\cap P^{1,1}}$,
 \item $\displaystyle \mathcal{H}^{n - 1, n - 1}_{\delbar} = \C \omega^{n - 1} \oplus L^{n - 2} \pa{\mathcal{H}^{1, 1}_{\del} \cap P^{1, 1}}$,
 \item $\displaystyle \mathcal{H}^{n - 1, n - 1}_{\del} = \C \omega^{n - 1} \oplus L^{n - 2} \pa{\mathcal{H}^{1, 1}_{\delbar} \cap P^{1, 1}}$.
\end{enumerate}
\end{cor}
\begin{proof}
The first decomposition follows from the one proved in Theorem \ref{thm: decomposition} by conjugation.

To prove the second, observe that the Hodge-$*$-operator induces an isomorphism $\mathcal{H}^{1, 1}_{\del} \simeq \mathcal{H}^{n - 1, n - 1}_{\delbar}$. Via this isomorphism, $\omega$ corresponds to $\omega^{n - 1}$, while by \eqref{*-primitive} on primitive $(1, 1)$-forms we have that $* = -\frac{1}{(n - 2)!}L^{n - 2}$. So we just have to apply $*$ to the decomposition of the previous point.

Finally, the last point follows from the second by conjugation.
\end{proof}

Recall that by \cite{cirici-wilson-2} (cf. also \cite{tardini-tomassini-diff-operators}) on compact almost-K\"ahler manifolds,
$$
\Delta_{\delbar}+\Delta_{\mu}=\Delta_{\del}+\Delta_{\bar\mu}
$$
and so, for every $p,q$
$$
\mathcal{H}^{p,q}_{\delbar}\cap \mathcal{H}^{p,q}_{\mu}=
\mathcal{H}^{p,q}_{\del}\cap \mathcal{H}^{p,q}_{\bar\mu}.
$$
In particular, if $J$ is integrable, namely $(X,J,g,\omega)$ is a compact K\"ahler manifold, one recovers the well known identities
$$
\Delta_{\delbar}=\Delta_{\del}
$$
and
$$
\mathcal{H}^{p,q}_{\delbar}=
\mathcal{H}^{p,q}_{\del}.
$$
Therefore, one could wonder if this last identity holds true, also in the non-integrable case, for some special bidegrees. More precisely, we want to show that the two primitive decompositions we obtained in Theorem \ref{thm: decomposition} and Corollary \ref{cor: decomposition} for $\mathcal{H}^{1, 1}_{\delbar}$ and $\mathcal{H}^{1, 1}_{\del}$ are not the same.

\section{Relations between $\Delta_{\delbar}$ and $\Delta_{\del}$}\label{section:relations-laplacians}

Let us start by considering the $4$-dimensional case.
Let $\alpha_{1, 1}$ be a primitive $(1, 1)$-form on an almost-K\"ahler $4$-dimensional manifold $X$. It follows from \eqref{*-primitive} that $*\alpha_{1, 1} = -\alpha_{1, 1}$, as a consequence we have the following


\begin{prop}\label{prop:dim4-equality}
Let $X$ be an almost-K\"ahler $4$-dimensional manifold. Then,
on $(1,1)$-forms we have
$$
\Delta_{\delbar_{\mid A^{1,1}}}=\Delta_{\del_{\mid A^{1,1}}}
$$
and in particular their kernels coincide
\[\mathcal{H}^{1, 1}_{\delbar} = \mathcal{H}^{1, 1}_{\del}.\]
\end{prop}
Notice that this follows also from \cite{cirici-wilson-2}, since on almost-K\"ahler manifolds $\Delta_{\delbar}+\Delta_{\mu}=\Delta_{\del}+\Delta_{\bar\mu}$ and on $(1, 1)$-forms on $4$-dimensional almost-K\"ahler manifolds $\Delta_{\mu}=\Delta_{\bar\mu}=0$.\\

We show now that in higher dimension the equality
$$
\Delta_{\delbar_{\mid A^{1,1}}}=\Delta_{\del_{\mid A^{1,1}}}
$$
does not hold in general.

\begin{ex}\label{ex:torus-different-laplacian}
Let $\mathbb{T}^6=\mathbb{Z}^6\backslash \mathbb{R}^6$ be the $6$-dimensional torus with $(x_1,x_2,x_3,y_1,y_2,y_3)$ coordinates on $\mathbb{R}^6$. Let $f=f(x_2)$ be a non-constant $\mathbb{Z}$-periodic function and we define the following non left-invariant almost-complex structure $J$ on $\mathbb{T}^6$,
$$
J\partial_{x_1}:=e^{-f}\partial_{y_1},\quad
J\partial_{x_2}:=\partial_{y_2},\quad
J\partial_{x_3}:=\partial_{y_3}.
$$
A global co-frame of $(1,0)$-forms is given by
$$
\Phi^1:=dx_1+i\,e^fdy_1,\quad
\Phi^2:=dx_2+i\,dy_2,\quad
\Phi^3:=dx_3+i\,dy_3\,.
$$
The structure equations are
$$
d\Phi^1=-\frac{1}{4}f'(x_2)\Phi^{12}-\frac{1}{4}f'(x_2)\Phi^{2\bar 1}-
\frac{1}{4}f'(x_2)\Phi^{1\bar 2}+\frac{1}{4}f'(x_2)\Phi^{\bar1\bar2}
$$
and $d\Phi^2=d\Phi^3=0$.
Then, the $(1,1)$-form
$$
\omega:=\frac{i}{2}e^{-f}\Phi^{1\bar 1}+\frac{i}{2}\Phi^{2\bar 2}+\frac{i}{2}\Phi^{3\bar 3}
$$
is a compatible symplectic structure, namely $(J,\omega)$ is an almost-K\"ahler structure on $\mathbb{T}^6$.\\
Notice now that by a direct computation
$$
\bar\mu\Phi^{1\bar 3}=\frac{1}{4}f'(x_2)\Phi^{\bar1\bar2\bar3}\neq 0
$$
and
$$
\mu\Phi^{1\bar3}=0.
$$
Therefore, from \cite{cirici-wilson-2}
$$
(\Delta_{\delbar}-\Delta_{\del})\Phi^{1\bar3}=-\bar\mu^*\bar\mu\Phi^{1\bar3}\neq 0.
$$
The last point follows either by direct computation or by noticing that
$$
\bar\mu^*\bar\mu\Phi^{1\bar3}\neq 0 \quad\iff\quad
\| \bar\mu\Phi^{1\bar3}\|^2\neq 0\quad\iff\quad
\bar\mu\Phi^{1\bar 3}\neq 0.
$$
\end{ex}
Another example is provided by the following $8$-dimensional nilmanifold with a left-invariant almost-K\"ahler structure.
\begin{ex}
We recall the following construction contained in \cite{debartolomeis-tomassini}.
Set
\[
\mathbb{H}(1,2):=\left\lbrace
\left[\begin{matrix}
1 & 0 & x_1 & z_1\\
0 & 1 & x_2 & z_2\\
0 & 0 & 1 & y\\
0 & 0 & 0 & 1
\end{matrix}\right]
\mid x_1,x_2,y,z_1,z_2\in \mathbb{R}\right\rbrace.
\]
Let $\Gamma$ be the subgroup of matrices with integral entries, let $X:=\Gamma\backslash \mathbb{H}(1,2)$ and define
$$
M:=X\times \mathbb{T}^3.
$$
Denoting with $u,v,w$ coordinates on $\mathbb{T}^3$ we consider the following left-invariant $1$-forms
$$
e^1:=dx_2,\quad
e^2:=dx_1,\quad
e^3:=dy,\quad
e^4:=du,\quad
$$
$$
e^5:=dz_1-x_1dy,\quad
e^6:=dz_2-x_2dy,\quad
e^7:=dv,\quad
e^8:=dw,\quad
$$
and the structure equations become
$$
de^1=de^2=de^3=de^4=de^7=de^8=0,\quad
de^5=-e^{23},\quad
de^6=-e^{13}\,.
$$
We define the following symplectic structure
$$
\omega:=e^{15}+e^{26}+e^{37}+e^{48}
$$
and we take the compatible almost-complex structure defined by the following co-frame of $(1,0)$-forms
$$
\psi^1:=e^1+i\,e^5,\quad
\psi^2:=e^2+i\,e^6,\quad
\psi^3:=e^3+i\,e^7,\quad
\psi^4:=e^4+i\,e^8.
$$
By direct computation
$$
d\psi^{1\bar4}=-\frac{i}{4}\psi^{23\bar4}-\frac{i}{4}\psi^{2\bar3\bar4}+
\frac{i}{4}\psi^{3\bar2\bar4}-\frac{i}{4}\psi^{\bar2\bar3\bar4},
$$
hence
$$
\mu\psi^{1\bar4}=0,\quad
\bar\mu\psi^{1\bar4}=-\frac{i}{4}\psi^{\bar2\bar3\bar4}.
$$
Therefore,
$$
(\Delta_{\delbar}-\Delta_{\del})\psi^{1\bar4}=(\Delta_{\bar\mu}-\Delta_{\mu})\psi^{1\bar4}=\bar\mu^*\bar\mu\psi^{1\bar4}\neq 0,
$$
proving that
$$
\Delta_{\delbar}\neq\Delta_{\del}
$$
on $(1,1)$-forms. However, one can show that their kernels coincide, namely $\mathcal{H}^{1, 1}_{\delbar} = \mathcal{H}^{1, 1}_{\del}$.
\end{ex}

\begin{rem}
We want to point out that finding explicit examples of almost-K\"ahler manifolds with $\Delta_{\delbar}\neq\Delta_{\del}$ seems to be not so obvious. In fact, we couldn't find any left-invariant example in dimension $6$.
\end{rem}

Even though, $\Delta_{\delbar_{\mid A^{1,1}}}\neq\Delta_{\del_{\mid A^{1,1}}}$ in general,
we wonder whether their kernels coincide. Before showing that this is not the case we notice that the equality $\mathcal{H}^{1, 1}_{\delbar} = \mathcal{H}^{1, 1}_{\del}$ is equivalent to $\mathcal{H}^{1, 1}_{\delbar} \cap P^{1, 1} = \mathcal{H}^{1, 1}_{\del} \cap P^{1, 1}$.

\begin{lemma}\label{lemma:equiv-harmonic-primitive}
Let $(X, J, g, \omega)$ be an almost-K\"ahler manifold. Then $\mathcal{H}^{1, 1}_{\delbar} = \mathcal{H}^{1, 1}_{\del}$ if and only if $\mathcal{H}^{1, 1}_{\delbar} \cap P^{1, 1} = \mathcal{H}^{1, 1}_{\del} \cap P^{1, 1}$.
\end{lemma}
\begin{proof}
We prove only the non-trivial implication. Let $\alpha_{1, 1} \in \mathcal{H}^{1, 1}_{\delbar}$, then we can decompose it as $\alpha_{1, 1} = c \omega + \beta_{1, 1}$ with $c \in \C$ and $\beta_{1, 1} \in \mathcal{H}^{1, 1}_{\delbar} \cap P^{1, 1}$. Now,
\[\Delta_{\del} \alpha_{1, 1} = c \cdot \Delta_{\del} \omega + \Delta_{\del} \beta_{1, 1} = 0 + 0 = 0,\]
so $\alpha_{1, 1} \in \mathcal{H}^{1, 1}_{\del}$. The other inclusion is similar.
\end{proof}

We observe the following 

\begin{lemma}
Let $(X^{2n}, J, g, \omega)$ be a $2n$-dimensional almost-K\"ahler manifold. Let $k:=p+q\leq n$ and let $\alpha \in P^{p, q}$. Then, 
$$
\delbar\alpha=0 \qquad\Longrightarrow\qquad \del^*\alpha = 0\,.
$$
Similarly,
$$
\del\alpha=0 \qquad\Longrightarrow\qquad \delbar^*\alpha = 0\,.
$$
\end{lemma}
\begin{proof}
By \eqref{*-primitive} we have that
\[
*\alpha = (-1)^{\frac{k(k+1)}{2}}\frac{i^{p-q}}{(n - k)!} \alpha \wedge \omega^{n - k}.
\]
Since $\omega$ is closed this readily implies that $\delbar *\alpha = 0$.\\
The same holds switching $\delbar$ and $\del$.
\end{proof}

\begin{lemma}
Let $(X, J, g, \omega)$ be an almost-K\"ahler manifold. Let $\alpha_{1, 1} \in \mathcal{H}^{1, 1}_{\delbar} \cap P^{1, 1}$. Then $d^*\alpha_{1, 1} = 0$.
\end{lemma}
\begin{proof}
Since $*\alpha_{1, 1}$ is a $(n - 1, n - 1)$-form, by the previous lemma we have that
\[d *\alpha_{1, 1} = (\del + \delbar) *\alpha_{1, 1} = \del *\alpha_{1, 1} + \delbar *\alpha_{1, 1} = 0.\]
\end{proof}

\begin{lemma}
Let $(X, J, g, \omega)$ be an almost-K\"ahler manifold. Let $\alpha_{1, 1} \in \mathcal{H}^{1, 1}_{\delbar} \cap P^{1, 1}$. Then $d \alpha_{1, 1}$, $\mu \alpha_{1, 1}$, $\del \alpha_{1, 1}$, $\delbar \alpha_{1, 1}$ and $\bar{\mu} \alpha_{1, 1}$ are primitive.
\end{lemma}
\begin{proof}
From the previous lemma and \eqref{*-primitive} we deduce that
\[0 = d *\alpha_{1, 1} = -\frac{1}{(n - 2)!} d(\alpha \wedge \omega^{n - 2}) = -\frac{1}{(n - 2)!} d\alpha \wedge \omega^{n - 2}.\]
So $d\alpha_{1, 1}$ is primitive and by decomposition in types we deduce that also $\mu \alpha_{1, 1}$, $\del \alpha_{1, 1}$, $\delbar \alpha_{1, 1}$ and $\bar{\mu} \alpha_{1, 1}$ are primitive.
\end{proof}

We finally show that, in general, on compact almost-K\"ahler manifolds
$$
\mathcal{H}^{1,1}_{\delbar}\neq\mathcal{H}^{1,1}_{\del}.
$$
By Lemma (\ref{lemma:equiv-harmonic-primitive}) this will be done using primitive forms.

\begin{ex}
Using the same notations as in Example \ref{ex:torus-different-laplacian} we consider $\mathbb{T}^6=\mathbb{Z}^6\backslash \mathbb{R}^6$. Let $g=g(x_3,y_3)$ be a function on $\mathbb{T}^6$.
We define an almost-complex structure $J$ setting as global co-frame of $(1,0)$-forms 
$$
\varphi^1:=e^gdx_1+i\,e^{-g}dy_1,\quad
\varphi^2:=dx_2+i\,dy_2,\quad
\varphi^3:=dx_3+i\,dy_3\,.
$$
The structure equations are
$$
d\varphi^1=V_3(g)\varphi^{3\bar1}-\bar V_3(g)\varphi^{\bar1\bar3}
$$
where $\left\lbrace V_1,V_2,V_3\right\rbrace$ is the global frame of vector fields dual to $\left\lbrace\varphi^1,\varphi^2,\varphi^3\right\rbrace$,
and $d\varphi^2=d\varphi^3=0$. Assume finally that $g$ satisfies $V_3(g) \neq 0$.\\
Then, the $(1,1)$-form
$$
\omega:=\frac{i}{2}\varphi^{1\bar 1}+\frac{i}{2}\varphi^{2\bar 2}+\frac{i}{2}\varphi^{3\bar 3}
$$
is a compatible symplectic structure, namely $(J,\omega)$ is an almost-K\"ahler structure on $\mathbb{T}^6$.\\
Notice now that,
$$
\delbar\varphi^{1\bar 2}=V_3(g)\varphi^{3\bar1\bar2}\neq 0
$$
namely, $\varphi^{1\bar2}\notin\mathcal{H}^{1,1}_{\delbar}$ but $\varphi^{1\bar2}\in\mathcal{H}^{1,1}_{\del}$. Indeed,
$\del\varphi^{1\bar2}=0$ and, since $\varphi^{1\bar2}$ is primitive and $\omega$ is closed,
$$
\delbar*\varphi^{1\bar2}=\delbar(-\omega\wedge\varphi^{1\bar2})=-\omega\wedge\delbar\varphi^{1\bar2}=-\omega\wedge\left(V_3(g)\varphi^{3\bar1\bar2}\right)=0.
$$
Hence, $\del^*\varphi^{1\bar2}=-*\delbar*\varphi^{1\bar2}=0$.
\end{ex}

\section{Primitive decompositions in dimension $6$}

Notice that in view of Propositions \ref{prop1: decomposition}, \ref{prop2: decomposition}, Theorem \ref{thm: decomposition} and Corollary \ref{cor: decomposition} we have a full primitive description of all $\delbar$-harmonic forms on compact $4$-dimensional almost-K\"ahler manifolds. It is therefore natural to ask what happens for bidegrees different from $(p,0)$, $(0,q)$, $(n,q)$, $(p,n)$, $(1,1)$ and $(n-1,n-1)$ in higher dimension. The first interesting dimension to consider is $6$ and in this case the only bidegrees left are $(2,1)$ and $(1,2)$.
Let us focus, for instance, on bidegree $(2,1)$. The primitive decomposition of forms is
\begin{equation*}
A^{2,1}(X)=P^{2,1}\oplus L\left(A^{1,0}(X)\right).
\end{equation*}
Passing to $\delbar$-harmonic forms, it follows that 
\[
\mathcal{H}^{2,1}_{\delbar}\supseteq\left(\mathcal{H}^{2,1}_{\delbar}\cap P^{2,1}\right)\oplus L\left(\mathcal{H}^{1,0}_{\delbar}\right),
\]
indeed, on compact almost-K\"ahler manifolds, for bidegree reasons and \cite{cirici-wilson-2} one has
$$
\mathcal{H}^{1,0}_{\delbar}=
\mathcal{H}^{1,0}_{\delbar}\cap
\mathcal{H}^{1,0}_{\mu}=
\mathcal{H}^{1,0}_{\del}\cap
\mathcal{H}^{1,0}_{\bar\mu}.
$$
Therefore it is natural to wonder whether such inclusion is indeed an identity.
In fact, this is not the case in general, as it is shown by the following

\begin{prop}
There exists a compact almost-K\"ahler $6$-dimensional manifold $(X,J,\omega)$ such that
\[
\mathcal{H}^{2,1}_{\delbar}\neq\left(\mathcal{H}^{2,1}_{\delbar}\cap P^{2,1}\right)\oplus L\left(\mathcal{H}^{1,0}_{\delbar}\right).
\]
\end{prop}
\begin{proof}
We refer to Example \ref{ex:no-2-1-decomposition} for the proof of this Proposition.
\end{proof}

First we need the following lemma that will allow us to work only with left-invariant forms.
\begin{lemma}\label{lemma:left-inv}
Let $X^6=\Gamma\backslash G$ be the compact quotient of a $6$-dimensional, connected, simply-connected Lie group by a lattice and let $(J,\omega)$ be a left-invariant almost-K\"ahler structure on $X$. Let $\eta\in A^{2,1}(X)$ be a left-invariant $(2,1)$-form on $X$ with primitive decomposition
\[
\eta=\alpha+L\beta.
\]
Then, $\alpha$ and $\beta$ are left-invariant.
\end{lemma}
\begin{proof}
Let $\eta\in A^{2,1}(X)$ be a left-invariant $(2,1)$-form on $X$. Its primitive decomposition is
\[
\eta=\alpha+L\beta,
\]
with $\alpha\in A^{2,1}(X)$ primitive, i.e., $L\alpha=0$ and $\beta\in A^{1,0}(X)$. Notice that $\beta$ is indeed primitive for bidegree reasons. 
We apply $L$ to the decomposition and obtain
$$
L\eta=L^2\beta.
$$
Since $\omega$ is left-invariant then $L\eta$, and so $L^2\beta$, are left-invariant.
Now, since $L^2:\Lambda^1\to\Lambda^5$ is an isomorphism at the level of the exterior algebra, it follows that also $\beta$ is left-invariant. As a consequence, since $L\beta$ and $\eta$ are left-invariant, it follows that also $\alpha$ is left-invariant.
\end{proof}

\begin{ex}\label{ex:no-2-1-decomposition}

Let $X$ be the Iwasawa manifold defined as the quotient $X:=\Gamma\backslash\mathbb{H}_3$ where
$$
\mathbb{H}_3:=\left\lbrace
\left[\begin{matrix}
1 & z_1 &   z_3 \\
0 & 1   & z_2\\
0 & 0 & 1\\
\end{matrix}\right]
\mid z_1,z_2,z_3\in\mathbb{C}
\right\rbrace
$$
and
$$
\Gamma:=\left\lbrace
\left[\begin{matrix}
1 & \gamma_1 &   \gamma_3 \\
0 & 1   & \gamma_2\\
0 & 0 & 1\\
\end{matrix}\right]
\mid \gamma_1,\gamma_2,\gamma_3\in\mathbb{Z}[\,i\,]
\right\rbrace\,.
$$
Then, setting $z_j=x_j+iy_j$, there exists a basis of
left-invariant $1$-forms $\left\lbrace e_i\right\rbrace$ on $X$ given by
$$
\left\lbrace
\begin{array}{lcl}
e^1 & =& dx_1\\
e^2 & = & dy_1\\
e^3 & = & dx_2\\
e^4 & = & dy_2\\
e^5 & = & dx_3-x_1dx_2+y_1dy_2\\
e^6 & = & dy_3-x_1dy_2-y_1dx_2
\end{array}
\right.\,.
$$
The following structure equations hold
$$
\left\lbrace
\begin{array}{lcl}
de^1 & =& 0\\
de^2 & = & 0\\
de^3 & = & 0\\
de^4 & = & 0\\
de^5 & = & -e^{13}+e^{24}\\
de^6 & = & -e^{14}-e^{23}
\end{array}
\right.\,.
$$
Let us consider the non integrable left-invariant almost-complex structure $J$ given by
\[
\phi^1=e^1+ie^6,\ \ \ \phi^2=e^2+ie^5,\ \ \ \phi^3=e^3+ie^4
\]
being a global coframe of $(1,0)$-forms. 
By a direct computation the structure equations become (cf. also \cite{tardini-tomassini-dim6})
\begin{align*}
4\,d\phi^1&=-\phi^{13}-i\phi^{23}+\phi^{1\bar3}+\phi^{3\bar1}-i\phi^{2\bar3}+i\phi^{3\bar2}+\phi^{\bar1\bar3}-i\phi^{\bar2\bar3},\\
4\,d\phi^2&=-i\phi^{13}+\phi^{23}-i\phi^{1\bar3}+i\phi^{3\bar1}-\phi^{2\bar3}-\phi^{3\bar2}-i\phi^{\bar1\bar3}-\phi^{\bar2\bar3},\\
d\phi^3&=0.
\end{align*}
Endow $(X,J)$ with the left-invariant almost-K\"ahler structure given by
\[
\omega=2(e^{16}+e^{25}+e^{34})=i(\phi^{1\bar1}+\phi^{2\bar2}+\phi^{3\bar3}).
\]


We want to find an element $\eta\in A^{2,1}(X)$ which is contained in $\mathcal{H}^{2,1}_{\delbar}$ but is not contained in 
\[
\left(\mathcal{H}^{2,1}_{\delbar}\cap P^{2,1}\right)\oplus L\left(\mathcal{H}^{1,0}_{\delbar}\right).
\]
Thanks to Lemma \ref{lemma:left-inv}
it is sufficient to work with left-invariant forms. Indeed
if we find
 $\eta\in \mathcal{H}^{2,1}_{\delbar}$ left-invariant that cannot be decomposed as $\eta=\alpha+L\beta$, with $\alpha\in\mathcal{H}^{2,1}_{\delbar}\cap P^{2,1}$ and $\beta\in\mathcal{H}^{1,0}_{\delbar}$, both left-invariant forms, then $\eta\notin \left(\mathcal{H}^{2,1}_{\delbar}\cap P^{2,1}\right)\oplus L\left(\mathcal{H}^{1,0}_{\delbar}\right).$

A long, but direct and straightforward computation, shows that the space of left-invariant $\delbar$-harmonic $(2,1)$-forms is
\[
\C<\phi^{13\bar1}+\phi^{23\bar2},\,
\phi^{13\bar2}+\phi^{23\bar1}-2i\phi^{23\bar2},\,
\phi^{13\bar3}+\phi^{23\bar3}>,
\]
while it is immediate to see that the space of left-invariant forms which are contained in $L\left(\mathcal{H}^{1,0}_{\delbar}\right)$ is
\[
\C<\phi^{13\bar1}+\phi^{23\bar2}>.
\]
Since, for instance, $L(\phi^{13\bar2}+\phi^{23\bar1}-2i\phi^{23\bar2})=-2iL(\phi^{23\bar2})\neq0$, it means that $\phi^{13\bar2}+\phi^{23\bar1}-2i\phi^{23\bar2}$ is not primitive. Therefore $\phi^{13\bar2}+\phi^{23\bar1}-2i\phi^{23\bar2}$ is a left-invariant, $\delbar$-harmonic $(2,1)$-form, but it is not contained in 
\[
\left(\mathcal{H}^{2,1}_{\delbar}\cap P^{2,1}\right)\oplus L\left(\mathcal{H}^{1,0}_{\delbar}\right).
\]
\end{ex}

\end{document}